\documentclass[12pt,a4paper]{amsart}
 \pdfoutput=1
 
 \usepackage{amsmath,amssymb,amsthm}
 \usepackage{bm}  
 \usepackage{mathtools}
\usepackage{tikz}
\usetikzlibrary{arrows}
 \usepackage{thmtools}
 \usepackage{float} 
 
 \usepackage{xcolor} 	
 \usepackage{hyperref}
 \hypersetup{
 	colorlinks,
     linkcolor={red!60!black},
     citecolor={green!60!black},
     urlcolor={blue!60!black},
 }
 \usepackage[maxbibnames=99]{biblatex}
 \usepackage[utf8]{inputenc}
  \usepackage{csquotes}
 \usepackage[T1]{fontenc}
 \usepackage{lmodern}
 \usepackage[babel]{microtype}
 \usepackage[english]{babel}
 
 \usepackage{geometry}
 \usepackage{enumitem}

\DeclareMathAlphabet{\mymathbb}{U}{BOONDOX-ds}{m}{n}

\theoremstyle{plain}
\newtheorem{theorem}{Theorem}[section]

\newtheorem{claim}[theorem]{Claim}
\newtheorem{corollary}[theorem]{Corollary}
\newtheorem{lemma}[theorem]{Lemma}

\newtheorem{observation}[theorem]{Observation}

\newtheorem{question}[theorem]{Question}
\newtheorem{conjecture}[theorem]{Conjecture}

\newtheorem*{theorem*}{Theorem}
\newtheorem*{corollary*}{Corollary}

\theoremstyle{definition}

\title{Hindrance from a wasteful partial linkage}

\author{Attila Jo\'{o}}
\thanks{Funded by the Deutsche Forschungsgemeinschaft (DFG, German Research Foundation) - Grant No. 513023562 and partially by NKFIH 
OTKA-129211}
\address{Attila Jo\'{o},
Department of Mathematics, University of Hamburg, Bundesstra{\ss}e 55 (Geomatikum), 20146 Hamburg, Germany}
\email{attila.joo@uni-hamburg.de}

\keywords{hindrance, linkage, web}
\subjclass[2020]{Primary: 05C63  Secondary: 05C20, 05C40}
\begin{document}
\begin{abstract}
Let $ D=(V,E) $ be a (possibly infinite) digraph and $ A,B\subseteq V $. A hindrance consists of an $ AB $-separator $ S $ 
together with a set of disjoint $ AS $-paths linking a proper subset of $ A $ onto $ S $. Hindrances and configurations 
guaranteeing the existence of hindrances play an essential role in the proof of the infinite version of Menger's theorem and are 
important in the context of certain open problems as well. This motivates the investigation of circumstances under which 
hindrances appear. In this paper we show that if there is a ``wasteful partial linkage'', i.e. a set $ \mathcal{P} $ of disjoint $ 
AB $-paths with fewer unused vertices in $ B $ than in $ A $, then there exists a hindrance.
\end{abstract}
\maketitle

\section{Introduction}
Let $ G=(A,B, E) $ be a bipartite graph of any cardinality and suppose that our task is to find a matching that covers $ A $. Then a hindrance is 
a set $ X\subseteq A  $ for which the neighbourhood $ N_G(X) $ of $ X $ can be matched into a proper subset of $ X $. We 
call $ G $ hindered if such a set $ X $ exists. Note 
that if $ X $ is infinite, this does not preclude the existence of a matching covering $ A $. Indeed, for example if $ 
G $ is a 
complete bipartite graph with $ \left|A\right|=\left|B\right|=\aleph_0 $, then every infinite $ X\subseteq A $ is a hindrance 
although every injection from $ A $ to $ B $ provides a matching covering $ A $. A necessary and sufficient condition of this manner for the 
matchability of $ A $ 
was discovered by Aharoni \cite[Theorem 1.8]{aharoni2009menger}. For more results in infinite matching theory we refer to 
\cite{aharoni1991infinite}.

Hindrances can be defined in a more general setting. The triple $ \mathcal{W}=(D,A,B) $ is a web if $ D=(V,E) $ 
is a digraph of any cardinality and $ A,B\subseteq V $ where the vertices in $ A $ (in $ B $) have no 
ingoing 
(outgoing) edges. Suppose we are tasked with finding a linkage of $ A $ to $ B $ in $ D $, i.e. a set 
$ \mathcal{P} $ of disjoint $ AB $-paths such that the set $ \mathsf{in}(\mathcal{P}) $ of the initial vertices of $ 
\mathcal{P} $ is the whole $ 
A $. If 
there is an $ X\subseteq A $ such that a proper subset of $ X $ can be linked onto an $ XB $-separator $ S $; that is, there exists a 
set $ \mathcal{H} $ of disjoint $ XB $-paths with $ \mathsf{in}(\mathcal{H}) \subsetneq X $ and with terminal points $ 
\mathsf{ter}(\mathcal{H})=S $, then we call the web $ \mathcal{W} $ hindered. Note that adding the vertices in $ 
A\setminus 
X $ to $ \mathcal{H} $ as trivial paths, we obtain a set $ \mathcal{H}' $ of disjoint paths linking a proper subset of $ A $ 
onto 
an $ AB $-separator, namely  $S':=S\cup (A\setminus X )$. Such an $ \mathcal{H}' $ is called a hindrance. 
Hindrances play a key role in the proof of the infinite version of Menger's theorem (\cite[Theorem 1.6]{aharoni2009menger}) 
and they are also essential in 
the context 
of some open 
problems such as Matroid intersection conjecture \cite[Conjecture 1.2]{aharoni1998intersection}. This motivates us to 
investigate 
circumstances under which hindrances appear. 

A partial 
linkage $ \mathcal{P} 
$ in $ \mathcal{W} $ is a set of disjoint $ AB $-paths. We call $ \mathcal{P} $ wasteful if $ \left|A\setminus \mathsf{in}(\mathcal{P})\right|> 
\left|B\setminus \mathsf{ter}(\mathcal{P})\right| $. Note that if 
$ \left|A\right|>\left|B\right| $, then  every partial linkage is wasteful. The main result of this paper reads as follows:

\begin{theorem}\label{thm: main}
If a web $ \mathcal{W}=(D,A,B) $ admits a wasteful partial linkage $ \mathcal{P} $, then it is hindered.
\end{theorem}
The proof is based on a natural adaptation of the techniques used in the proof of the infinite version of König's theorem
(see \cite{aharoni1984konig}). We derive as a consequence 
the following statement:

\begin{corollary}\label{cor: main corollary}
If a bipartite graph $ G=(A,B,E) $ admits a matching $ M $ with $ \left|A\setminus V(M)\right|>\left|B\setminus V(M)\right| 
$, then it is hindered.
\end{corollary}

The paper is organized as follows. We introduce the necessary notation in the following section. In Section 
\ref{sec: proof}, we prove 
Theorem \ref{thm: main} and Corollary \ref{cor: main corollary}. In the last section (Section \ref{sec: outlook}), we 
conjecture a matroidal version of  
Theorem \ref{thm: main} and a possible simplification of the proof of the infinite version of Menger's theorem.

\section{Notation}
Ordered pairs $ \left\langle u, v\right\rangle $ in the context of directed edges are written as $ uv $. For a cardinal $ \kappa 
$ a $ \kappa 
$\emph{-set} 
is a set of cardinality $ \kappa $. We write $ \omega $ for the set of  natural numbers. A digraph is an ordered pair $ D=(V,E) $ 
where $ E \subseteq (V\times 
V)\setminus \{vv:\ v\in V \} $ and $ uv\in E $ implies $ vu\notin E $.\footnote{This restriction of digraphs is just for 
convenience. Subdividing each edge by a new vertex makes digraphs simple, while the existence of a wasteful linkage and of 
a hindrance are both invariant under this operation.} Fix a digraph $ D=(V,E) $. The \emph{in-neighbourhood} $ 
\boldsymbol{N^{-}(U)} $ of a set $ U\subseteq V $ consists 
of those $ v\in 
V\setminus U $ 
for which there is a $ w\in U $ with $ vw\in E $.  

A trivial  
\emph{trail} is a single vertex $ v\in V $. A non-trivial trail is a finite, nonempty sequence $T= (v_0v_1, 
 v_1v_2, \dots v_{n-1}v_n) $  where $ v_iv_{i+1}\in E $ and $ v_iv_{i+1}\neq v_jv_{j+1} $ for $ i\neq j $. We write $ 
 \boldsymbol{\mathsf{in}(T)} $ for the initial vertex $ v_0 $ and $ \boldsymbol{\mathsf{ter}(T)} $ for the terminal vertex $ 
 v_n $ of the trail  (while $ \mathsf{in}(T)=\mathsf{ter}(T)=v $ if $ T=v $ is a trivial trail). 
 Furthermore, $ V(T) 
 $ and $ E(T)$ stand for the vertex set and edge set of 
 the subgraph corresponding to $ T $.  We say that $ T $ is an $ XY $-trail if it is either trivial and an element of $ X\cap Y $ 
 or non-trivial with $ v_0\in X $, $ v_n\in Y $ and $ v_i \notin 
 X\cup Y $ for $ 0<i<n $.  A \emph{path} $ P $ is a trail that is either trivial or non-trivial and has the additional property of not repeating vertices, 
 i.e.  $ v_i\neq v_j $ 
 for $ i\neq j $.    For a set $ 
\mathcal{T} $ of trails we let $ \boldsymbol{\mathsf{in}(\mathcal{T})}:=\{ \mathsf{in}(T):\ T\in \mathcal{T} \} $ and $ 
\boldsymbol{\mathsf{ter}(\mathcal{T})}:=\{ 
\mathsf{ter}(T):\ T\in \mathcal{T} \} $, moreover $ \boldsymbol{V(\mathcal{T})}=\bigcup_{T\in \mathcal{T}}V(T) $  and $ 
\boldsymbol{E(\mathcal{T})}=\bigcup_{T\in \mathcal{T}}E(T) $. We call a set $ \mathcal{T} $ of trails \emph{disjoint} if $ V(R)\cap 
V(T)=\emptyset  $ for every distinct $ R, T\in \mathcal{T} $ and  $ v $\emph{-joint} if $ \mathsf{ter}(T)=v $ 
for each $ T\in 
\mathcal{T} $ and $ V(R)\cap V(T)=\{ v \} $ for every distinct $ R, T\in \mathcal{T} $. A set $ \mathcal{P} $ of disjoint 
paths \emph{links} $ X $ \emph{onto} $ Y $ if $ \mathsf{in}(\mathcal{P})=X $ and $ \mathsf{ter}(\mathcal{P})=Y $.

 A \emph{web} is a triple $\boldsymbol{ \mathcal{W}}=(D,A,B) $ where $ D=(V,E) $ is a digraph,  $ A,B\subseteq V $ 
 and the vertices in $ A $ ($ B $) have no ingoing (outgoing) edges. 
 An $ XY $-\emph{separator}  is an $ S\subseteq V $ that meets all $ XY $-paths. A \emph{hindrance}  is a set $ 
  \mathcal{H} $ of disjoint $ AS $-paths where $ S 
 $ is an $ AB $-separator and $ \mathcal{H} $ links a proper 
 subset of $ A $ onto $ S $.
 A 
 \emph{partial linkage} of $ \mathcal{W} $ is a set $ \mathcal{P} $ of disjoint  $ AB $-paths. 
 
 A \emph{partially linked 
 web} is a 
 quadruple $ 
 \boldsymbol{\mathcal{L}}=(D,A,B, \mathcal{P}) $ where $\mathcal{W}= (D,A,B) $ is a web and $ \mathcal{P} $ is a 
 partial linkage in it. We write $ 
\boldsymbol{ \widehat{A}} $ and $ \boldsymbol{\widehat{B}} $ for $ A \setminus \mathsf{in}(\mathcal{P}) $ and $ B 
\setminus 
 \mathsf{ter}(\mathcal{P}) $ respectively (we add $ \mathcal{L} $ as subscript if it is not clear from the context).   
 
 The \emph{alternating trails} in $ \mathcal{L} $ consist of the trivial trails corresponding to the elements of $ \widehat{A} 
 $ as well as those
 $T= (v_0v_1, 
 v_1v_2, \dots v_{n-1}v_n) $ for which

 \begin{enumerate}
 \item\label{item: reversed edges} $ T $ is a non-trivial trail in the digraph $ \boldsymbol{D^{*}}=(V,\boldsymbol{E^{*}}) 
 $ we 
  obtain from $ D $ by reversing 
  the edges in $ E(\mathcal{P})$, 
 \item\label{item: start in A hat}  $ v_0\in \widehat{A} $,
 \item\label{item: repeat vertex} if $ v_i=v_j $ for $ i< j $, then $ v_i\in V(\mathcal{P}) $ and $ j\neq n $,
 \item\label{item go back rule} if $ v_{i}\in V(\mathcal{P}) $,   $v_{i-1}v_{i}\in E^{*}\cap E $ and  $ i<n $, then 
 $ 
 v_{i}v_{i+1}\in E^{*}\setminus E $. 
 \end{enumerate}
(Property (\ref{item go back rule}) states that upon arriving at a vertex $ v\in V(\mathcal{P}) $ by using a ``forward'' edge implies that the 
next edge we use must be the unique ``backward'' edge from $ v $ unless we terminate at $ v $.)
 
We say that the alternating trails $ R $ and $ T $ are \emph{strongly disjoint} if they are disjoint and there is no $ P\in 
\mathcal{P} $ 
such that $ V(R)\cap V(P)\neq \emptyset  $ and $ V(T)\cap V(P)\neq \emptyset  $.  An alternating trail is \emph{augmenting} 
if it terminates in $ \widehat{B} $.    Let $ 
\boldsymbol{O} $ consist 
 of those $ v\in \widehat{B} $ for which there exists an $ 
\left|\widehat{A}\right|$-set  $ \mathcal{T} $ of $ v $-joint augmenting trails and let 
$\boldsymbol{U}:=\widehat{B}\setminus 
O $. We call these the \emph{popular} and \emph{unpopular} vertices respectively.

\section{Proof of Theorem \ref{thm: main}}\label{sec: proof}
Fix a partially linked web $ \mathcal{L}=(D,A,B,\mathcal{P}) $  and let  $ \mathcal{W}:=(D,A,B) $. 
\subsection{Preparations}
\begin{observation}\label{obs: dis to strongly dis}
For every infinite cardinal $ \kappa $, every $ \kappa $-set of disjoint alternating trails includes a $ \kappa $-subset of 
strongly 
disjoint alternating trails.
\end{observation}
\begin{proof}
The desired $ \kappa $-subset of strongly 
disjoint alternating trails can be built by a straightforward transfinite recursion. Indeed, let $ \mathcal{T} $ be a $ \kappa $-set of disjoint 
alternating trails.  In step $ \alpha<\kappa $ we have already chosen strongly 
disjoint alternating trails $ \mathcal{T}_\alpha=\{ T_\beta:\ \beta<\alpha \} $. The trails in $ \mathcal{T}_\alpha $ meet at most 
$ \left|\alpha\right| \cdot \aleph_0<\kappa $ paths in $ \mathcal{P} $. Since $ \left|\mathcal{T}\right|=\kappa $ and the trails in $ \mathcal{T} $ are 
disjoint, we can pick a $ T_\alpha $ from $ \mathcal{T} $ that is disjoint from all of these paths as well as from all the trails 
in $ \mathcal{T}_\alpha $. 
\end{proof}

The proof of Menger's theorem  based on augmenting trails (sometimes called augmenting walks)  is well known, so we do 
not 
provide detailed proofs for the following two lemmas. For detailed proofs in the context of undirected graphs see 
\cite[Lemma 3.3.2 and 3.3.3]{diestel2017graph}. \footnote{For the directed version we could not find a suitable 
reference. Textbooks are typically reducing vertex-Menger to edge-Menger where the corresponding 
``augmenting path 
lemmas'' are simpler. }
\begin{lemma}\label{lem: no aug hind}
If  $ \widehat{A}\neq \emptyset  $ and there is no augmenting trail, then $ \mathcal{W} $ is hindered. 
\end{lemma}
\begin{proof}
For each $ P\in \mathcal{P} $, let $ v_P $ be the last vertex on $ P $ for which there is an alternating trail $ T $ with $ 
\mathsf{ter}(T)=v_P $ and let $ v_P:=\mathsf{in}(P) $ if no alternating trail meets $ P $. One 
can show (see \cite[Lemma 3.3.3]{diestel2017graph}) that $ S:=\{ v_P:\ P\in \mathcal{P} \} $ is an 
$ 
AB $-separator. The initial segment 
of  each $ P\in \mathcal{P} $ up to $ v_P $   provides a hindrance. 
\end{proof}

\begin{lemma}\label{lem: augmenting mass}
If $ \mathcal{T} $ is a set of strongly disjoint augmenting trails, then there is a partial linkage $ \mathcal{Q} $ with $ 
\mathsf{in}(\mathcal{Q})=\mathsf{in}(\mathcal{P}) \cup 
\mathsf{in}(\mathcal{T})$ and $ \mathsf{ter}(\mathcal{Q})=\mathsf{ter}(\mathcal{P}) \cup \mathsf{ter}(\mathcal{T})$.
\end{lemma}
\begin{proof}
For $ T\in \mathcal{T} $, let $ \mathcal{P}_T $ be the set of the finitely many paths in $ \mathcal{P} $ meeting $ T $. 
Then one can 
find (see \cite[Lemma 3.3.2]{diestel2017graph}) a set $ \mathcal{Q}_T $ of $ \left|\mathcal{P}_T\right|+1 $ many $ AB 
$-paths with $  V(\mathcal{Q}_T)\subseteq V(\mathcal{P}_T)\cup 
V(T) $ and $ \mathsf{in}(T), \mathsf{ter}(T)\in V(\mathcal{Q}_T) $. Since the sets $ V(\mathcal{P}_T)\cup 
V(T)  $ for $ T\in \mathcal{T} $ are pairwise disjoint, the lemma follows.
\end{proof}

The following lemma handles the special case of Theorem \ref{thm: main} where $\left|\widehat{A}\right| $ is infinite and each vertex in 
$\widehat{B}\setminus 
\widehat{A}$ is popular.
\begin{lemma}\label{lem: everybody popular}
If  $ \left|\widehat{A}\right|> \left|\widehat{B}\right|$, $ \left|\widehat{A}\right|\geq \aleph_0 $ and $ 
O=\widehat{B}\setminus \widehat{A}$, then a proper subset of $ A $ can be linked onto $ B $. In particular, $ \mathcal{W} 
$ is hindered. 
\end{lemma}
\begin{proof}
 By Lemma \ref{lem: augmenting mass}, it is enough to find a set $ 
\mathcal{T} $ of strongly disjoint 
augmenting trails with $ \mathsf{ter}(T)=\widehat{B} $. For each $v\in  \widehat{B}\cap \widehat{A} $, let the
trivial augmenting trail $ v $ be in $ \mathcal{T} $. By assumption, for each 
$ v \in \widehat{B}\setminus \widehat{A} $ there is an $ \left|\widehat{A}\right| $-set $ \mathcal{T}_v $ of $ v $-joint 
augmenting trails.  Thus the rest of  $ 
\mathcal{T} $ can be constructed by a straightforward transfinite recursion. 

Indeed, let $ \lambda:=\left|\widehat{B}\setminus 
\widehat{A}\right| $ and let $ \widehat{B}\setminus 
\widehat{A}=\{ v_\alpha:\ \alpha<\lambda \} $ be an enumeration. In step $ \alpha<\lambda $ we have already chosen 
strongly 
disjoint alternating trails $ \mathcal{T}_\alpha=\{ T_\beta:\ \beta<\alpha \} $ where   $ 
V(T_\beta)\cap \widehat{A}\cap \widehat{B}=\emptyset  $ and $ \mathsf{ter}(T_\beta)=v_\beta $ for $ \beta<\alpha $. The 
trails in $ 
\mathcal{T}_\alpha $ meet at most 
$ \left|\alpha\right| \cdot \aleph_0<\left|\widehat{A}\right| $ paths in $ \mathcal{P} $ and we know that $ \left|\widehat{A}\cap 
\widehat{B}\right|<\left|\widehat{A}\right| $.  Therefore we can pick a $ T_\alpha \in \mathcal{T}_{v_\alpha} $  that is 
disjoint from: the trails in $ 
\mathcal{T}_\alpha $, those paths in $ \mathcal{P} $ that meet a trail in $ \mathcal{T}_\alpha $, and  $ \widehat{A}\cap 
\widehat{B} $.
\end{proof}

From now on, we assume that $ \mathcal{P} $ is wasteful w.r.t. $ \mathcal{W} $, i.e. $ \left|\widehat{A}\right| 
>\left|\widehat{B}\right|$.  Suppose 
first that $ \left|\widehat{A}\right|\leq\aleph_0 $. Then $ n:=\left|\widehat{B}\right|<\aleph_0 $. We apply induction on $ n $. 
If $ n=0 $, then $ \mathcal{P} $ itself is a hindrance. Assume that $ n>0 $. If there is no augmenting trail, then we are done 
by Lemma \ref{lem: no aug hind}. If there is an augmenting trail $ T $, then we apply Lemma \ref{lem: augmenting mass} 
with $ \mathcal{T}=\{ T \} $. For the resulting $ \mathcal{Q} $, let $ \mathcal{L}':=(D,A,B, \mathcal{Q}) $. Then  $ 
\left|\widehat{A}_{\mathcal{L}'}\right|=\left|\widehat{A}_{\mathcal{L}}\right|-1 $ and $ 
\left|\widehat{B}_{\mathcal{L}'}\right|=\left|\widehat{B}_{\mathcal{L}}\right|-1=n-1 $ according to Lemma \ref{lem: 
augmenting mass}. Since $ \mathcal{W}_{\mathcal{L}'}=\mathcal{W}_{\mathcal{L}}  $, we are done by applying the 
induction 
hypothesis for $ n-1 $.

\subsection{Elimination of unpopular vertices in $ \omega $ steps}
Now we can assume that $  \aleph_1\leq \left|\widehat{A}\right|=:\kappa  $.  We can also assume, without loss of generality, 
that $ \kappa$ is regular. Indeed, otherwise we pick an uncountable regular cardinal $ \lambda $ with $ \kappa > \lambda 
>\left|\widehat{B}\right| $ and delete all but $ \lambda $ many vertices in $ \widehat{A} $. Any hindrance $ \mathcal{H} $ 
of the resulting partially linked web can be extended to a hindrance of the original by adding the deleted vertices to $ 
\mathcal{H} $ as trivial paths.

We 
define recursively  partially linked webs $ \mathcal{L}_n=(D_n,A,B_n,\mathcal{P}_n) $ for  $ n<\omega $ (see Figure 
\ref{fig: figure}). Let 
 $ \mathcal{L}_0:=\mathcal{L} $.  Suppose  that $ 
\mathcal{L}_n $ is already 
defined. Let us denote $ U_n:=U_{\mathcal{L}_n} $ and $ O_n:=O_{\mathcal{L}_n} $, as well as $ 
\widehat{A}_n:=\widehat{A}_{\mathcal{L}_n} $ and $ 
\widehat{B}_n:=\widehat{B}_{\mathcal{L}_n} $. We define  $ 
B_{n+1}:=(B_n\setminus (U_n\setminus A))\cup 
N_{D_n}^{-}(U_n) $. Observe that $ B_{n+1} $ separates $ B_n $ from $ A $ in $ D_n $. To define $ 
D_{n+1}=(V,E_{n+1}) $, 
we define  $ E_{n+1} $  from $ E_n $ by the deletion 
of the outgoing edges of the vertices in $ N_{D_n}^{-}(U_n) $. Finally, we let $ \mathcal{P}_{n+1} $ consist of the 
initial segments of the paths in  $ \mathcal{P}_n $ up to their first vertex in $ B_{n+1} $. 

\begin{figure}[h]
\centering

\begin{tikzpicture}[thick,scale=0.75, every node/.style={transform shape}]

\draw  (-4.5,5.5) rectangle (5.5,3.5);
\draw  (-7,0.5) rectangle (5.5,-1.5);
\draw  plot[smooth, tension=.7] coordinates {(-3,5.5) (-3,3.5)};
\draw  plot[smooth, tension=.7] coordinates {(-1,5.5) (-1,3.5)};

\draw[->]  plot[smooth, tension=.7] coordinates {(0,-0.5) (-0.5,1) (0,2) (-0.5,2.5)};
\draw[->]   plot[smooth, tension=.7] coordinates {(1,-0.5) (0.5,1) (1,1.5) (0.5,2.5)};
\draw[->]   plot[smooth, tension=.7] coordinates {(2,-0.5) (2.5,1) (2,1.5) (2.5,2) (2,3) (2,4.5)};
\draw[->]  plot[smooth, tension=.7] coordinates {(3,-0.5) (3.5,1.5) (3,2.5) (3,4.5)};
\draw[->]  plot[smooth, tension=.7] coordinates {(4,-0.5) (4.5,1.5) (4,2.5) (4,4.5)};


\draw[dashed, ->]  plot[smooth, tension=.7] coordinates {(-0.5,2.5) (-1,3) (0,4.5)};
\draw[dashed, ->]  plot[smooth, tension=.7] coordinates {(0.5,2.5) (0,3) (1,4.5)};

\node at (-3.8,5.8) {$O_n$};
\node at (-2,5.8) {$U_n$};
\node at (-2.8,6.2) {$\widehat{B}_n$};
\node at (-3.8,0.8) {$\widehat{A}$};
\node at (-1,1) {$A$};
\node at (0.2,6.2) {$B_n$};
\node at (-0.4,3) {$W_n$};

\node (v2) at (-2.4,3.8) {};
\node (v6) at (-2,3.8) {};
\node (v4) at (-1.8,3.8) {};
\node (v11) at (-1.4,3.8) {};
\node (v1) at (-3.4,-0.4) {};
\draw[->]  (v1) edge (v2);
\node (v3) at (-1,3) {};
\draw[->]  (v3) edge (v4);
\node (v5) at (-2,2.6) {};
\draw[->]  (v5) edge (v6);
\node (v9) at (-1.8,3.6) {};
\node (v8) at (-2.2,3.8) {};
\node (v7) at (-2.4,2.2) {};
\draw[->]  (v7) edge (v8);
\node at (-1.6,3.8) {};
\node (v10) at (-0.4,2.4) {};
\draw[->]  (v10) edge (v11);
\node (v12) at (0.6,2.4) {};
\node (v13) at (-1.2,3.8) {};
\draw[->]  (v12) edge (v13);
\node (v14) at (-1.6,2.6) {};
\draw[->]  (v14) edge (v4);
\node at (1.4,1.8) {$\mathcal{P}_n$};
\end{tikzpicture}
\caption{A step of the recursion. (Drawing $A$ and 
$B$ as disjoint is done purely for clarity.)} \label{fig: figure}
\end{figure}
\begin{observation}\label{obs: A hat unchanged}
$ \widehat{A}_{n}=\widehat{A} $ for each $ n<\omega $.
\end{observation}
\begin{observation}\label{obs: separates n}
For every $ n<\omega $, $ B_n $ separates $ \bigcup_{k\leq n}B_k $ (in particular $ B $) from $ A $ in $ D $.
\end{observation}
\begin{proof}
For $ n=0 $ this is trivial. Suppose we already know for some $ n $ that $ B_n $ separates 
$ \bigcup_{k\leq n}B_k $ from $ A $ in $ D $. By construction $ B_{n+1} $ separates $ B_n $ from $ A $ in $ D_n 
$. Thus it follows by induction and by $ E_n\subseteq E $ that  $ B_{n+1} $ separates $ \bigcup_{k\leq n+1}B_k $ from $ A 
$ in $ D_n 
$. Since
every edge in $ E\setminus E_n $ is an outgoing edge of a vertex in $ \bigcup_{k\leq n}B_k $,  $ B_{n+1} $ separates $ \bigcup_{k\leq n+1}B_k 
$ from $ A $ in $ D 
$ as well.
\end{proof}

Let $\boldsymbol{W_n}:= V(\mathcal{P}_{n})\setminus 
V(\mathcal{P}_{n+1})  $.
\begin{lemma}\label{lem: remains alternating trail}
For every $ n<\omega $, the following holds:
\begin{enumerate}
\item Every $ \mathcal{L}_{n} $-alternating trail $ T $ with 
$ V(T)\cap (W_n \cup N_{D_{n}}^{-}(U_n))\subseteq \{ \mathsf{ter}(T) \} $ is an $ \mathcal{L}_{n+1} $-alternating 
trail.
\item Every $ \mathcal{L}_{n+1} $-alternating trail $ T $ with $ V(T)\cap W_n \subseteq \{ \mathsf{ter}(T) \} $ is an $ 
\mathcal{L}_n $-alternating trail.
\end{enumerate}
\end{lemma}
\begin{proof}
The lemma follows directly from the definitions by checking formally the properties (\ref{item: reversed edges})-(\ref{item 
go back rule}) in the definition of alternating trails.

 Both statements are true for trivial trails 
because  $ \widehat{A}_n=\widehat{A}_{n+1} $ by Observation \ref{obs: A hat 
unchanged}. Suppose that $ T=(v_0v_1,\dots, v_{n-1}v_n) $  is a non-trivial $ \mathcal{L}_{n} $-alternating trail with 
$ V(T)\cap (W_n \cup N_{D_{n}}^{-}(U_n)) \subseteq \{ \mathsf{ter}(T) \} $.  Since each edge in  $ E_n^{*}\setminus 
E_{n+1}^{*} $ is  an outgoing edge of  a vertex in $ (W_n \cup N_{D_{n}}^{-}(U_n)) $, no such edge is used by $ T 
$. Thus $ 
T $ is a trail in $ D^{*}_{n+1} $, hence $ T $ satisfies (\ref{item: reversed edges})  w.r.t. $ \mathcal{L}_{n+1} $. Clearly,  
$ 
\mathsf{in}(T)\in 
\widehat{A}_{n}=\widehat{A}_{n+1} $, 
thus (\ref{item: start in A hat}) remains true. Since $ V(\mathcal{P}_{n+1})\subseteq V(\mathcal{P}_n) $ and $ T $ does not 
repeat vertices that are not 
in $ V(\mathcal{P}_{n}) $,  it does not repeat vertices not in $ V(\mathcal{P}_{n+1}) $ either. Therefore (\ref{item: repeat 
vertex})  holds. Since $ T $ 
respects (\ref{item go back rule}) w.r.t. $ \mathcal{P}_n $ and $ T $ has 
at most its terminal vertex in 
$ W_n  $, no violation of  (\ref{item go back rule}) can occur.

Suppose that $ T=(v_0v_1,\dots, v_{n-1}v_n) $  is a non-trivial $ \mathcal{L}_{n+1} $-alternating trail with $ V(T)\cap 
W_n 
\subseteq \{ \mathsf{ter}(T) \} 
$. Note that if an  $ e\in E_{n+1}^{*} $ is not in $ E_{n}^{*}$,  then there is a $ P\in \mathcal{P}_n $ with  $e\in E(P) $ such that the 
initial 
segment $ P' $ of $ P $ that is in $  \mathcal{P}_{n+1} $ does not 
use $ e $. In other words, $ e $ is reversed in $ D^{*}_n $ but not any more in $ D^{*}_{n+1} $. Hence each edge in $ E_{n+1}^{*}\setminus 
E_{n}^{*} $ is an outgoing edge of a vertex in $ W_n $.   Thus no such edge is used by $ T 
$ and therefore $ 
T $ is a trail in $ D^{*}_{n} $, hence (\ref{item: reversed edges}) holds.  Clearly,  $ \mathsf{in}(T)\in 
\widehat{A}_{n+1}=\widehat{A}_{n} $, 
thus (\ref{item: start in A hat}) holds. We claim that no vertex in $ V(\mathcal{P}_n) $ is repeated by $ T $.  By definition, $ V(\mathcal{P}_n)= 
V(\mathcal{P}_{n+1})\cup W_n 
$ and no vertex in $ V(\mathcal{P}_n) $ is repeated.  Furthermore, only the terminal vertex of $ T $ may be in $ W_n $ but 
terminal vertices are never repeated by the definition of alternating trails. Thus (\ref{item: repeat vertex}) holds. Finally, $ T $ respects (\ref{item 
go back rule}) corresponding to $ \mathcal{P}_{n+1} $ and at most its terminal vertex is in $ W_n $, thus (\ref{item 
go back rule}) holds corresponding to $ \mathcal{P}_n $ as well.
\end{proof}
\begin{lemma}\label{lem: key lemma}
For every $ n<\omega $:
\begin{enumerate}[label=(\roman*)]
\item[$ (i_n) $]\label{item: no kappa aug} there is no $ \kappa $-set $ \mathcal{T} $ of disjoint $ \mathcal{L}_n $-augmenting trails;
\item[$ (ii_n) $]\label{item: few meets Wn} there is no $ \kappa $-set $ \mathcal{T} $ of disjoint $ \mathcal{L}_n 
$-alternating $ \widehat{A}W_n$-trails.
\end{enumerate}
\end{lemma}
\begin{proof}
We apply induction on $ n $. Clearly, $ (i_0) $ holds because $ \left|\widehat{B}_0\right|<\kappa $ (by $ 
\left|\widehat{B}\right|<\left|\widehat{A}\right|=\kappa $) and every $ 
\mathcal{L}_0 $-augmenting trail
terminates in $ 
 \widehat{B}_0 $ by definition. Suppose that we already know $ (i_n) $ for some $ n $.   Next we prove $ (ii_n) $ and $ 
  (i_{n+1})$.
\begin{claim}\label{clm: no lot trail to Un}
 There is no $ \kappa $-set of $ \mathcal{L}_n 
  $-augmenting trails with $ \mathsf{ter}(\mathcal{T})\subseteq U_n $ where any two trails in $ \mathcal{T} $ are either disjoint or sharing 
   exactly their terminal vertex.
\end{claim}
\begin{proof}
If $ \left|\mathsf{ter}(\mathcal{T}) \right|<\kappa$, then, by the regularity of $ \kappa $, there is 
 a 
 $ v\in \mathsf{ter}(\mathcal{T}) $ such that $ \kappa $ many $ T\in \mathcal{T} $ terminates at $ v $, contradicting  $v\in U_n $. Thus $ 
 \left|\mathsf{ter}(\mathcal{T})\right|=\kappa $. But then by choosing for each $ v\in \mathsf{ter}(\mathcal{T}) $ exactly one $ T\in 
 \mathcal{T} $ with $ \mathsf{ter}(T)=v $, we obtain a $ \kappa $-set of disjoint $ \mathcal{L}_n 
 $-augmenting trails which contradicts $ (i_n) $.
\end{proof}

 Suppose for a contradiction that $ (ii_n) $ fails, i.e. there is a $ \kappa $-set $ \mathcal{T} $ of disjoint $ \mathcal{L}_n 
 $-alternating $ 
 \widehat{A}W_n $-trails. By Observation \ref{obs: dis to strongly dis}, we may assume that the trails in $ \mathcal{T} $ 
 are strongly disjoint. For $ 
 T\in \mathcal{T} $, let 
 $ P_T $ be the unique path in $ \mathcal{P}_n $ with $ \mathsf{ter}(T)\in V(P_T) $ and let $ v_P $ be the terminal vertex of 
 the unique path in $ \mathcal{P}_{n+1} $ that is a proper initial segment of $ P_T $ (exists by the definition of $ W_n $). 
 Let $ T' $ be a forward-extension of $ 
 T $ that 
 we obtain by going backwards along $ P_T $ until  $ v_P $  and then adding an outgoing edge of $ v_P $ in $ D_n $
 with head in $ U_n $ 
 (which exists because $ v_P $ is nothing else than the first vertex of $ P_T $ that is in $ N_{D_{n}}^{-}(U_n) $). Then $ 
 \mathcal{T}':=\{ T':\ T\in \mathcal{T} \} $ is a $ \kappa $-set of $ \mathcal{L}_n 
 $-augmenting paths with $ \mathsf{ter}(\mathcal{T}')\subseteq U_n $ where any two trails in $ \mathcal{T}' $ are either disjoint or sharing 
 exactly their terminal vertex which contradicts Claim \ref{clm: no lot trail to Un}.
 
 Suppose for a contradiction that $ (i_{n+1}) $ fails, i.e. there is a $ \kappa $-set $ \mathcal{T} $ of disjoint $ 
 \mathcal{L}_{n+1} 
 $-augmenting trails. By $( i_n) $ we know that less than $ \kappa $ many $ T\in \mathcal{T} $ are $ \mathcal{L}_n 
 $-augmenting trails, thus by 
 deleting these we can assume that no $ T\in \mathcal{T} $ is an $ \mathcal{L}_n $-augmenting trail. then exactly one of the following holds for 
 each
  $ T\in \mathcal{T} $: either $ T $ is not even an $ \mathcal{L}_n $-alternating trail or it is but $ 
 \mathsf{ter}(T)\in B_{n+1}\setminus B_n $. Since one of these two cases occurs $ \kappa $ often, we can assume that one of them holds 
 actually for every $ T\in \mathcal{T} $. Suppose first, that no trail in $ \mathcal{T} $ is $ \mathcal{L}_n $-alternating. 
 Then  (2) of Lemma 
 \ref{lem: remains alternating trail}   guarantees that each $ 
 T\in \mathcal{T} $ meets $ W_n $. Let $ \mathcal{T}' $ consists of the initial segments of the trails in $ \mathcal{T} $ up to their first vertex in $ 
 W_n $. Then (2) 
 of Lemma \ref{lem: remains alternating trail} ensures that $ 
 \mathcal{T}' 
 $ consists of $ \mathcal{L}_n $-alternating trails. Since $ \left|\mathcal{T}'\right|=\kappa $, this contradicts $ (ii_n) $. Assume now that every 
 $ T\in \mathcal{T} $ is $ \mathcal{L}_n $-alternating but terminates in $ B_{n+1}\setminus B_n= N_{D_n}^{-}(U_n) $. Then 
 every $ T\in 
 \mathcal{T} $ can be extended forward by a new edge to  an $ \mathcal{L}_n $-augmenting trail $ T' $ that terminates in $ 
 U_n $. But then the set 
 $ \mathcal{T}' $  of these extensions contradicts Claim \ref{clm: no lot trail to Un}.
\end{proof}
\begin{corollary}\label{cor: few arrives to A}
$ \left|\widehat{B}_n\cap \widehat{A}\right|<\kappa $ for each $ n<\omega $.
\end{corollary}
\begin{proof}
Since $ \widehat{A}=\widehat{A}_n $, each $ v\in \widehat{B}_n\cap \widehat{A} $ counts as a trivial $ 
\mathcal{L}_n $-augmenting trail. Thus the statement follows from $ (i_n) $ of Lemma \ref{lem: key lemma}.
\end{proof}
Now we show that the popularity of the vertices are preserved during the recursion.
\begin{corollary}\label{cor: less than k lose}
Suppose that $ v\in O_n $ and let $ \mathcal{T} $ be a witness for this, i.e.  a $ \kappa $-set  of $ v $-joint $ \mathcal{L}_n 
$-augmenting trails. Then there are only less than $ \kappa $ many $ T\in \mathcal{T} $ that are not $ \mathcal{L}_{n+1} 
$-augmenting trails.
\end{corollary}
\begin{proof}
 Let $ 
\mathcal{T}' $ consists of those $ T\in \mathcal{T} $  that are not $ \mathcal{L}_{n+1} $-augmenting trails. Suppose for a 
contradiction that $ \left|\mathcal{T}'\right|=\kappa $.  Then  (1) of Lemma \ref{lem: remains alternating trail} ensures that each $T\in 
\mathcal{T}' $ meets  $ W_n \cup N_{D_{n}}^{-}(U_n) $. Let $ \mathcal{T}'' $ consist of the initial segments of the trails in $ \mathcal{T}' $ 
until the first vertex in $ W_n \cup N_{D_{n}}^{-}(U_n) $. If $ \kappa $ many $ T\in \mathcal{T}'' $ terminates in $ W_n $, then we get a 
contradiction to $ (ii_n) $ of Lemma \ref{lem: key lemma}. If $ \kappa $ many $ T\in \mathcal{T}'' $ terminates in $ N_{D_{n}}^{-}(U_n) $, 
then we extend them forward by a new edge to reach $ U_n $ and get a contradiction to Claim \ref{clm: no lot trail to Un}.
\end{proof}
\begin{corollary}\label{cor: On grows}
$ O_n\subseteq O_{n+1} $ for each $ n<\omega $.
\end{corollary}
\begin{proof}
Suppose that  $v\in  O_n $ and let $ \mathcal{T} $ be a witness for this. As earlier, let $\mathcal{T}' $ consists of those $ 
T\in \mathcal{T} $  that are not $ \mathcal{L}_{n+1} $-augmenting trails. Then $ \left|\mathcal{T}'\right|<\kappa $ by 
Corollary \ref{cor: less than k lose}. Hence $ \left|\mathcal{T}\setminus \mathcal{T}'\right|=\kappa $ and therefore $ 
\mathcal{T}\setminus \mathcal{T}' $ witnesses $ v\in O_{n+1} $.
\end{proof}

\begin{corollary}\label{cor: On small}
$ \left|O_n\right|<\kappa $ for each $ n<\omega $.
\end{corollary}
\begin{proof}
We apply induction on $ n $. For $ n=0 $ the statement holds because $ O_0\subseteq \widehat{B}_0 $ and $ 
\left|\widehat{B}_0\right|<\kappa $ by assumption. Assume we already know that $ \left|O_n\right|<\kappa $ for some $ 
n<\omega $.  Suppose for a 
contradiction that $ \left|O_{n+1}\right| \geq \kappa $. Then, by the induction hypotheses we must have $ 
\left|O_{n+1}\setminus O_n\right| \geq \kappa $.  Let  $ C $ be a $ \kappa  $-subset of $ O_{n+1}\setminus O_n $. By using 
the
definition of $ O_{n+1} $, we can find by a straightforward transfinite recursion a $ \kappa $-set $ \mathcal{T} $ of disjoint $ 
\mathcal{L}_{n+1} $-alternating trails with $ \mathsf{ter}(\mathcal{T})=C $. Note that $  
O_{n+1}\setminus O_n \subseteq 
\widehat{B}_{n+1}\setminus\widehat{B}_{n}\subseteq N^{-}_{D_n}(U_n)  $ by construction. Assume first that $ \kappa $ many $ 
T\in 
\mathcal{T} $ 
meets $ W_n $. Then by taking suitable initial 
segments of these trails, we obtain a $ \kappa $-set $ \mathcal{T}' $ of $ \mathcal{L}_{n+1} $-alternating 
$ \widehat{A}W_n $-trails. By (2) of Lemma \ref{lem: remains alternating trail}, these are also $ \mathcal{L}_n $-alternating 
trails. But this contradicts  $ (ii_n) $ of Lemma \ref{lem: key lemma}. Thus we may assume that no $ T\in \mathcal{T} $ 
meets $ W_n $. By (2) of Lemma \ref{lem: remains alternating trail}, each $ T\in \mathcal{T} $ is an $ \mathcal{L}_n 
$-alternating 
trail as well. Since $ \mathsf{ter}(\mathcal{T})\subseteq  N^{-}_{D_n}(U_n) $, we can extend each $ T\in \mathcal{T} $ 
forward to terminate in $ U_n $. But then the existence of the resulting set of $ \mathcal{L}_n $-augmenting trails 
contradicts Claim \ref{clm: no lot trail to Un}.
\end{proof}

\subsection{The limit partially linked web}
Let $ 
E_\omega:=\bigcap_{n<\omega}E_n $ and $ D_\omega:=(V,E_\omega) $. We define  
$ B_\omega:=\bigcup_{m<\omega} \bigcap_{m<n<\omega} B_n $ and $ \mathcal{P}_\omega:=\bigcup_{m<\omega} 
\bigcap_{m<n<\omega} 
\mathcal{P}_n $.

 In each iteration, each path in the partial linkage was replaced 
by its initial segment. Roughly speaking,  $ 
 \mathcal{P}_\omega $  consists of the shortest versions of each path that we get during the process. Similarly, the sets $ B_n $ are ``getting closer 
 and closer'' to $ A $ in each iteration. The 
 vertices that arrive at $ A $ 
 or become popular at some stage form the set $ B_\omega $.
 
 Let $  \widehat{A}_\omega:=\widehat{A}_{\mathcal{L}_\omega}$, $  
 \widehat{B}_\omega:=\widehat{B}_{\mathcal{L}_\omega}$  and $ O_\omega:=O_{\mathcal{L}_\omega} $.
\begin{observation}\label{obs: hat omega}
$ \widehat{A}_\omega=\widehat{A} $.
\end{observation}
\begin{observation}
$ \mathcal{P}_\omega $ is a partial linkage in $ \mathcal{W}_\omega:=(D_\omega,A,B_\omega) $.
\end{observation}
\begin{lemma}\label{lem: B omega in A small}
$ \left|\widehat{B}_\omega \cap \widehat{A}\right|<\kappa $.
\end{lemma}
\begin{proof}
By construction $ \widehat{B}_\omega \cap \widehat{A}=\bigcup_{n<\omega}(\widehat{B}_n \cap \widehat{A})  $ and we 
know that    $ \left|\widehat{B}_n \cap 
\widehat{A}\right|<\kappa $ for each $ n<\omega $ (see Corollary \ref{cor: few arrives to A}). Since $ 
\kappa $ is a 
regular uncountable cardinal, we conclude that $\left| \widehat{B}_\omega \cap \widehat{A}\right|= 
\left|\bigcup_{n<\omega}(\widehat{B}_n \cap \widehat{A})\right|<\kappa $. 
\end{proof}
\begin{lemma}\label{lem: B omega separates}
$ B_\omega $ separates $ B $ from $ A $ in $ D $.
\end{lemma}
\begin{proof}
Let $ P $ be an $ AB $-path in $ D $. By Observation \ref{obs: separates n},  $ P $ meets $ B_n $ for each $ n $. Let $ 
v $ be 
the 
first vertex of $ P $ in $ 
\bigcup_{n<\omega}B_n $. Then there is an $ m<\omega $ with $ v\in B_m $. For 
every $ k \geq m $ the set $ B_k $ separates $ B_m $ from $ A $ (see Observation \ref{obs: separates n}). By the choice of $ v $ this implies that 
we must have $ v\in B_k $ for every $ k \geq m $. But then by definition $ v\in B_\omega $.
\end{proof}

\begin{lemma}\label{lem: B hat minus A hat}
$ \widehat{B}_\omega\setminus \widehat{A}=\bigcup_{n<\omega}O_n $.
\end{lemma}
\begin{proof}
Suppose that $ v\in  \widehat{B}_\omega\setminus \widehat{A}$. Then by the definition of $ \widehat{B}_\omega $  there 
is an $m<\omega $ such that $ 
v \in 
\widehat{B}_n\setminus 
\widehat{A} $ for $ n\geq m $. By definition $  \widehat{B}_m= U_m\cup O_m
 $,  therefore $ v\in U_m\cup O_m $. But $ v \notin U_m $ since otherwise $ v\in U_m\setminus \widehat{A} $ and 
 therefore $ 
v\notin \widehat{B}_{m+1} $ which contradicts the fact that $ v \in 
\widehat{B}_n\setminus 
\widehat{A} $ for $ n\geq m $. Hence $ v \in O_m\subseteq  \bigcup_{n<\omega}O_n$.

Assume that $ v\in \bigcup_{n<\omega}O_n $. Then there is an $ m<\omega $ with $ v\in O_m $. It follows that $ v\notin \widehat{A} $ because 
$ O_m\cap \widehat{A}=\emptyset $ is evident from the definition of popular vertices and the fact that $ \left|\widehat{A}\right|=\kappa>1 $. By 
Corollary \ref{cor: On grows} we know that $ v\in O_n 
\subseteq 
\widehat{B}_n $ 
for each $ n \geq m $. But then by definition $ v \in \widehat{B}_{\omega} $, thus $ v\in \widehat{B}_\omega\setminus \widehat{A} $.
\end{proof}
\begin{corollary}\label{cor: B hat minus A hat omega small}
$ \left|\widehat{B}_\omega\setminus \widehat{A}\right|<\kappa $
\end{corollary}
\begin{proof}
We have just seen that $ \widehat{B}_\omega\setminus 
\widehat{A}=\bigcup_{n<\omega}O_n $ (Lemma \ref{lem: B hat minus A hat}). Furthermore,   $ \left|O_n\right|<\kappa $ 
for each $ 
n<\omega 
$ (see Corollary \ref{cor: On small}). As $ 
\kappa $ is a 
regular uncountable cardinal, we conclude that $  \left|\widehat{B}_\omega\setminus 
\widehat{A}\right|=\left|\bigcup_{n<\omega}O_n\right|<\kappa $.
\end{proof}
\begin{observation}\label{obs: augmenting at the end}
If $ T $ is an $ \mathcal{L}_n $-augmenting trail for every large enough $ n $, then $ T $ is an $ \mathcal{L}_\omega 
$-augmenting trail.
\end{observation}
\begin{proof}
From the construction of $ \mathcal{L}_\omega $ it is clear that if  any of the properties (\ref{item: reversed edges})-(\ref{item go back rule}) is 
violated by $ T $ w.r.t. $ \mathcal{L}_\omega $, then it is violated w.r.t. $ \mathcal{L}_n $ for every large enough $ n $. 
\end{proof}
\begin{lemma}\label{lem: popular at the end}
$ \widehat{B}_{\omega}\setminus 
\widehat{A}=O_\omega $.
\end{lemma}
\begin{proof}
The inclusion ``$ \supseteq $'' holds because $ \widehat{B}_{\omega}=O_\omega\cup U_\omega $ by definition and $ O_\omega \cap 
\widehat{A}=\emptyset $. To show the inclusion ``$ \subseteq $'', let $ v\in \widehat{B}_{\omega}\setminus 
\widehat{A} $ be given. Since $ \widehat{B}_\omega\setminus \widehat{A}=\bigcup_{n<\omega}O_n $ by Lemma \ref{lem: B hat minus A 
hat}, there is an $ m<\omega $ with $ v\in O_m $. Let $ \mathcal{T}_m $ be a witness for $ v\in O_m $, i.e. a  $ \kappa $-set of $ v $-joint  $ 
\mathcal{L}_m $-augmenting trails. We define by recursion $ \mathcal{T}_n $ for $ n \geq m $. If $ \mathcal{T}_n $ is 
already defined, let $ \mathcal{T}_{n+1} $ consists of those $ 
T\in \mathcal{T}_n $ that are also $ \mathcal{L}_{n+1} $-augmenting trails. Finally, let $ \mathcal{T}_\omega:= 
\bigcap_{n 
\geq m}\mathcal{T}_n 
$. By Observation \ref{obs: augmenting at the end}, $ \mathcal{T}_\omega $ consists of $ \mathcal{L}_\omega 
$-augmenting trails.   We have $ \left|\mathcal{T}_m\right|=\kappa $ by definition and Corollary
\ref{cor: less than k lose} ensures that $ \left|\mathcal{T}_n\setminus \mathcal{T}_{n+1}\right|<\kappa $ for each $ n \geq m 
$.  Since $ \kappa $ is a regular uncountable cardinal, it follows that $ \left|\mathcal{T}_\omega\right|=\kappa $. But then $ 
\mathcal{T}_\omega $ witnesses $ v\in O_\omega $.
\end{proof}
We intend to apply Lemma \ref{lem: everybody popular} to $ \mathcal{L}_\omega $. We have $ 
\widehat{A}_\omega=\widehat{A} $ by Observation \ref{obs: hat omega} where $ \left|\widehat{A} \right|=\kappa$ by 
definition. Since
$ \left|\widehat{B}_\omega \cap \widehat{A}\right|<\kappa $ (Lemma \ref{lem: B omega in A small}) and $ 
\left|\widehat{B}_\omega\setminus 
\widehat{A}\right|<\kappa $ (Corollary \ref{cor: B hat minus A hat omega small}), 
we conclude $ \left|\widehat{B}_\omega\right|<\kappa $. Finally, $ \widehat{B}_{\omega}\setminus 
\widehat{A}=O_\omega $ by Lemma \ref{lem: popular at the end}. It follows that $ 
\mathcal{L}_\omega $ satisfies the conditions of Lemma \ref{lem: everybody popular}. Thus there is a hindrance $ 
\mathcal{H} $ in $ 
\mathcal{W}_\omega:=(D_\omega,A,B_\omega) $ that links a proper subset of $ A 
$ onto $ B_\omega $ in $ D_\omega $. Since $ D_\omega $ is a subgraph of $ D $ and  $ B_\omega $
separates $ B $ 
from $ A $ in $ D $ as well (Lemma \ref{lem: B omega separates}), this $ \mathcal{H} $ is a hindrance in $ \mathcal{W} $ 
too. This completes the proof of Theorem \ref{thm: main}.
\subsection{Proof of Corollary \ref{cor: main corollary}}
Let $ D $ be the orientation of $ G $, in which all the
edges are pointing towards $ B $. Note that the vertex covers of $ G $ are exactly the $ AB $-separators of $ D $. Then a matching $ M $ defines a 
partial linkage $ \mathcal{P} $ in the web $ \mathcal{W}:=(D,A,B) $. The assumption on $ M $ translates to $ \mathcal{P} 
$ being wasteful.   
Then Theorem \ref{thm: main}  ensures that  there 
 is a hindrance $ \mathcal{H} $ in $ \mathcal{W} $. Let  $ 
\mathcal{H}' $ consist of the non-trivial 
paths in $ \mathcal{H} $. Then each $ P\in \mathcal{H}' $ has exactly one edge and $ 
M':=E(\mathcal{H'}) $ is a matching.  Moreover, for $ X:=\mathsf{in}(\mathcal{H}')\cup (A\setminus V(\mathcal{H})) $ we 
have $ 
N_G(X)=\mathsf{ter}(\mathcal{H}') $ because $ \mathsf{ter}(\mathcal{H}) $ is a vertex cover in $ G $. Since $ 
\mathcal{H} $ is a hindrance, we know that $ A\setminus 
V(\mathcal{H})\neq \emptyset  $, therefore  $ 
M' $ matches $ N_{G}(X) $ to a proper subset of $ X $ and thus we conclude that $ X $ is a hindrance.
\section{Outlook}\label{sec: outlook}
Hindrances can be defined in the context of further problems such as the previously mentioned Matroid intersection 
conjecture. For the 
conjecture and the 
corresponding matroidal terminology one may refer to authors such as:  Aharoni \& Ziv \cite{aharoni1998intersection}, 
Bowler \& Carmesin 
\cite{bowler2015matroid} and ourselves 
\cite{joo2021MIC}.   

Suppose that $ M $ and 
$ N $ are finitary matroids on $ E $ and our task is to find an $ M 
$-independent spanning set of $ N $. An $ I $ is an  $ (M,N) $\emph{-wave} if  $I\in \mathcal{I}(M)\cap 
\mathcal{I}(N.\mathsf{span}_M(I)) $.  An $ I $ is an $ (M,N) $\emph{-hindrance} if it an $ (M,N) $\emph{-wave}
that does not span $ N.\mathsf{span}_M(I) $ (see \cite[Definition 3.2]{aharoni1998intersection}). We call $ (M,N) $ 
hindered if there exists an $ (M,N) $-hindrance.

\begin{conjecture}
If $ M $ and $ N $ are finitary matroids on a common ground set such that  there is an $ I\in \mathcal{I}(M) \cap 
\mathcal{I}(N) $ with $ r(M/I)<r(N/I) $, then $ (M,N)$ is hindered.
\end{conjecture}

Transforming a so-called $ \kappa $-hindrance (see \cite[p. 35-36]{aharoni2009menger}) to a hindrance is an essential part of 
the proof of the infinite version of Menger's theorem.  Together with an auxiliary tool called ``Bipartite conversion'', it occupies two entire chapters 
in \cite{aharoni2009menger}.

\begin{question}
Is it possible to prove the existence of a hindrance assuming the existence of a $ 
\kappa $-hindrance (\cite[Theorem 7.30]{aharoni2009menger}) directly along the lines of our 
proof of Theorem \ref{thm: main} avoiding bipartite conversion?
\end{question}

\printbibliography
\end{document}